\documentclass{template_arxiv}

\usepackage{amsmath,amsfonts,amssymb,amsthm}
\usepackage{graphicx}
\usepackage[dvipsnames]{xcolor}
\usepackage{geometry}
\usepackage{bm} 

\usepackage{enumitem}
\setlist{ 
    itemsep=0.2\baselineskip, 
    topsep=0.2\baselineskip, 
    parsep=0pt, 
    partopsep=0pt 
}

\usepackage[ 
    colorlinks=true, 
    allcolors=blue!60!black 
]{hyperref}

\usepackage{microtype} 

\usepackage{xspace}

\newtheorem{lemma}{Lemma}
\newtheorem{theorem}[lemma]{Theorem}
\newtheorem{proposition}[lemma]{Proposition}
\newtheorem{corollary}[lemma]{Corollary}
\newtheorem{conjecture}[lemma]{Conjecture}
\newtheorem{question}{Question}

\theoremstyle{definition}

\newtheorem{remark}[lemma]{Remark}
\newtheorem{example}[lemma]{Example}

\newcommand{\N}{\mathbb{N}}
\newcommand{\Z}{\mathbb{Z}}
\newcommand{\Q}{\mathbb{Q}}
\newcommand{\R}{\mathbb{R}}
\newcommand{\Acal}{\mathcal{A}}
\newcommand{\Fcal}{\mathcal{F}}

\newcommand{\ab}{\bm{\mathrm{ab}}}
\newcommand{\ac}{\rho}
\newcommand{\card}{\mathrm{card}}
\newcommand{\GL}{\mathrm{GL}}
\newcommand{\Hbal}{$(\mathcal{H}_{\mathrm{bal}})$\xspace}
\newcommand{\Hcomp}{$(\mathcal{H}
_{\mathrm{ab}})$\xspace}
\newcommand{\Hfreq}{$(\mathcal{H}_{\mathrm{RI}})$\xspace}
\newcommand{\lan}{\mathcal{L}}
\newcommand{\pref}{\mathrm{pref}}
\newcommand{\Span}{\mathrm{Span}}

\title{A proof of Rauzy’s conjecture on abelian complexity}

\newcommand{\shorttitle}{A proof of Rauzy’s conjecture on abelian complexity}

\date{June 14, 2026}

\author[,1,2]{M\'elodie Andrieu \thanks{melodie.andrieu@univ-littoral.fr}}
\author[,1,3]{L\'eo Vivion \thanks{lvivion.math@gmail.com}}

\affil[1]{Laboratoire de Math\'ematiques Pures et Appliqu\'ees Joseph Liouville, Universit\'e du Littoral C\^ote d'Opale, France.}

\affil[2]{Centro de Modelamiento Matem\'atico, Universidad de Chile and IRL-CNRS 2807, Chile.}

\affil[3]{Department of Mathematics, University of Li\`ege, Li\`ege, Belgium}

\newcommand{\shortauthor}{M. Andrieu, L. Vivion}

\begin{document}

\maketitle

\begin{abstract}
We resolve a conjecture posed by Rauzy in 1983 concerning the $d$-ary generalizations of Sturmian words. A classical theorem by Coven and Hedlund from 1973 states that Sturmian words, which are classically defined by their subword complexity (a combinatorial refinement of  topological entropy), or as the natural codings of irrational rotations on the circle, are also characterized by their abelian complexity (a combinatorial refinement of the notion of discrepancy).
More precisely, Sturmian words are exactly the binary infinite words with rationally independent letter frequencies and minimal abelian complexity, which is in this case constant and equal to $2$. We prove that there exist no infinite ternary words with rationally independent letter frequencies and constant abelian complexity equal to $3$, thereby establishing Rauzy's conjecture.

\vspace*{.5cm}
\end{abstract}

\section{Introduction}

This article is concerned with a conjecture by Rauzy which emerged in the context of generalizing Sturmian words (which are binary words) to alphabets of arbitrary size. Sturmian words are a central object in combinatorics on words and symbolic dynamics. They were introduced in 1940 by Morse and Hedlund \cite{MH40}, and have been  studied continuously ever since (see the lecture notes \cite{Loth,Pytheas02,And26}, and, for a recent result, the article \cite{Ficial16}).
A striking feature of Sturmian words is that they admit various equivalent and insightful descriptions from perspectives as diverse as combinatorics, dynamical systems, arithmetic, and geometry. From the combinatorial perspective (which is the focus of this paper), they are classically defined through their subword complexity.

By definition, the \emph{subword complexity} of an infinite word $w$ (where an \emph{infinite word} is a sequence indexed by $\mathbb{N}$ taking values in a finite set called an \emph{alphabet}) is the function that counts, for every positive integer $n$, the number of distinct subwords of length $n$ that appear in $w$. For instance, the subword complexity of the periodic word $w = 1212121212\ldots$ is the constant function $n \mapsto 2$ (indeed, for each $n$, one finds exactly two subwords of length $n$: one that starts with the letter $1$, and one that starts with the letter $2$). In their seminal work from 1938, Morse and Hedlund proved that some kind of converse holds.

\begin{theorem}[Morse and Hedlund, 1938 \cite{MH38}]\label{th:MH38}
An infinite word $w$ is eventually periodic if and only if its subword complexity  is eventually constant, if and only if there exists $n \in \N_{>0}$ such that $w$ admits at most $n$ subwords of length $n$.
\end{theorem}

The subword complexity measures how complicated an infinite word is, that is, how well the knowledge of small portions of the infinite word provides knowledge of larger portions. The subword complexity can be understood as a combinatorial refinement of the dynamical notion of topological entropy \cite[Chapter 4]{LindMarcus}.
A consequence of the Morse--Hedlund Theorem~\ref{th:MH38} is that the  subword complexity of non-eventually periodic words is at least $n \mapsto n+1$. In their subsequent paper \cite{MH40}, Morse and Hedlund proved that  infinite words with subword complexity exactly $n \mapsto n+1$ exist; these words are called \emph{Sturmian words}.  It follows from their definition that Sturmian words are written over a two-letter alphabet (these letters are indeed the two subwords of length $1$). The following word, called the \emph{Fibonacci word}, is a classical example of a Sturmian word:
\begin{equation}\label{eq:Fibo}w_{Fibo} = 0100101001001010010100100101001001\ldots\end{equation}
Combinatorially, the Fibonacci word can be constructed as the limit of the iterates of the substitution $\sigma: 0 \mapsto 01, 1 \mapsto 0$. (A \emph{substitution} is a morphism on the free monoid.) Dynamically, it is obtained as the symbolic coding of the trajectory of the point $\alpha = 2- \varphi$ (where $\varphi$ denotes the golden ratio) under iteration of the rotation of angle $\alpha$ on the unit circle $\R/\Z$,
with respect to the partition $[0,1)= [0,1-\alpha)\cup[1-\alpha,1)$. Note that every Sturmian word can also be constructed as the limit of the iterates of suitably chosen substitutions (\cite{MH40}; see also \cite{GLR09} for an alternative presentation), as well as the symbolic trajectory of an irrational rotation of the circle \cite{MH40}. (This second description of Sturmian words will arise in the final step of the proof of Rauzy's conjecture.) 

In this article, we are interested in another well-known combinatorial characterization of Sturmian words: as the non-eventually periodic words with minimal abelian complexity.

The notion of abelian complexity traces back to the works of Coven and Hedlund from 1973. By definition, two finite words are \emph{abelian equivalent} if they are anagrams (for example, \texttt{astronomer} $\sim_{ab}$ \texttt{moonstarer}). The \emph{abelian complexity} of an infinite word is the function that counts, for each $n \in \mathbb{N}_{>0}$, the number of pairwise non-abelian equivalent subwords of length $n$ that appear in $w$:
\[\begin{array}{rccl}
\rho_w \colon &\N_{>0} &\longrightarrow &\N \\
&n &\longmapsto &\card\; \lan_n(w)/_{\sim_{ab}}.
\end{array}\]
(In this expression, $\lan_n(w)$ denotes the set of subwords of length $n$ appearing in $w$.) For instance, the abelian complexity of the periodic word $w = 121212121212\ldots$ is the function 
\[\rho_w(n)=\begin{cases} 1 \; \text{ if $n$ is even}, \\ 2 \; \text{ otherwise.} \end{cases} \]
The abelian complexity is a refinement of the combinatorial notion of balancedness as well as its dynamical counterpart, the discrepancy, which measure how fast the distribution of letters in growing prefixes of $w$ converges to its limit value \cite{Ada03}.  Beyond the question of computing abelian complexities, numerous ``abelian problems''---some of them inherited from the work of Erd\H{o}s---have been considered in combinatorics on words (see, for example, the recent survey \cite{FP23}).

In 1973, Coven and Hedlund proved that the minimality of the abelian complexity characterizes Sturmian words.

\begin{theorem}
[Coven and Hedlund, 1973, \cite{CH73}]\label{th:CH73}
Sturmian words are exactly the non-eventually periodic words with constant abelian complexity  equal to $2$. This abelian complexity is furthermore minimal in the following sense: if there exists a positive integer $n$ for which all length-$n$ subwords of an infinite word $w$ are abelian equivalent, then $w$ is eventually periodic (in fact, it is even purely periodic).
\end{theorem}

\begin{remark}
Although the notions of subword and abelian complexities may seem, at first, minor variations of one another, they \textbf{are not} closely related (in particular, the Coven--Hedlund Theorem \ref{th:CH73} is not easy to prove). For instance, it is possible to find, already on the ternary alphabet:
\begin{itemize}
    \item infinite words whose abelian complexity is bounded above by $3$, but whose subword complexity grows exponentially (such a construction  follows from \cite[Theorem 4.3]{RSZ11}; see Example~\ref{ex:RSZ}).
    \item infinite words with linear subword complexity and an unbounded abelian complexity (see, for example, \cite{CFZ00,And21}).
\end{itemize}

It is thus remarkable that the minimality of both the subword and abelian complexities \emph{independently} characterizes Sturmian words.
\end{remark}

\medskip
In a series of papers and seminars in the 1980s, Rauzy initiated the study of different generalizations of Sturmian words over the ternary alphabet. This led, notably, to the introduction of the \emph{Tribonacci word}  and the \emph{Rauzy fractal} \cite{rau82}, and, later on, of the whole class of \emph{Arnoux--Rauzy words} (\cite{Rau8283,AR91}), which can be understood as a combinatorial and continued-fraction-wise generalization of Sturmian words. It also led to the introduction of \emph{cubic billiard words} \cite{Rau8283,AMST94} and \emph{natural codings of interval exchanges} \cite{Rau7677,Rau79}, which generalize Sturmian words from a dynamical--geometrical perspective. These famous papers contain numerous conjectures and questions, some of which have been completely answered while others remain unsolved. (For example, let us mention that the computation of the subword complexity of cubic billiard words  was established over the period 1994--2009 by the successive works of Arnoux--Mauduit--Shiokawa--Tamura, Baryshnikov, and Bédaride \cite{AMST94, Bar95, Bed09}, while the characterization of the language of interval exchange transformations was given by Ferenczi and Zamboni in 2008 \cite{FZ08}.)

In this article, we are interested in the following conjecture.

\begin{conjecture}[Rauzy, 1983, \cite{Rau8283}, Section 6.2] \label{conj:Rauzy}
Except for very particular vectors of letter frequencies, there exist no infinite ternary words with constant abelian complexity equal to $3$.
\end{conjecture}
Let us recall that the frequency of the letter $a$ in an infinite word $w$ is the limit, if it exists, of the proportion of occurrences of the letter $a$ in longer and longer prefixes of $w$:
\begin{equation}\label{eq:def_frequencies}
f_w(a) = \lim_{n \to \infty}\, \frac{\vert \pref_n(w)\vert_a}{n}.
\end{equation}
(In the expression above, $\vert \pref_n(w)\vert_a$ denotes the number of occurrences of the letter $a$  among the first $n$ letters of $w$.)
It is well known and not difficult to verify that this limit exists for every infinite word whose abelian complexity is bounded.

Unfortunately, Rauzy did not specify what a ``very particular'' vector is. Clearly, Conjecture \ref{conj:Rauzy} admits trivial counterexamples. For example, the  following two ternary words have an abelian complexity constant equal to $3$:
\begin{equation}\label{eq:counterexamples}\begin{array}{lll} w_1& =&21000000000000000000000000000000000000000000 \ldots,\\
w_2& =& 2\cdot w_{Fibo} = 20100101001001010010100100101001001 \ldots 
\end{array}\end{equation}
(Indeed, for every $n \geq 1$, $w_1$ admits exactly three subwords of length $n$, and they are pairwise non-abelian equivalent. For $w_2$, it follows from Theorem \ref{th:CH73} that the Fibonacci word \eqref{eq:Fibo} contains exactly two abelian classes of subwords of length $n$, which are written using only the letters $0$ and $1$, and to which we must thus add the abelian class of the length-$n$ prefix of $w_2$, which contains the letter $2$.) Note that $w_2$ is not eventually periodic.

We believe that Rauzy already had these counterexamples in mind, perhaps together with others, but could not determine the proper domain of validity of his conjecture. And indeed, in 2011, two additional remarkable families of counterexamples were constructed by Richomme, Saari and Zamboni \cite{RSZ11}. Contrary to $w_1$ and $w_2$, these counterexamples are uniformly recurrent (an infinite word is said \emph{recurrent} if each of its subwords appears infinitely often, it is \emph{uniformly recurrent} if they furthermore appear with bounded gaps); in particular, they contain infinitely many occurrences of every letter.

\medskip
In this article, we formalize Rauzy's unspecified condition as: ``with rationally independent letter frequencies'' (by rational independence, we mean that the frequencies of the letters span a linear space of maximal dimension over the field of rational numbers). It is already satisfactory that the set of such vectors of letter frequencies has Lebesgue measure 1 in the set of all vectors of letter frequencies. But most importantly, this choice is natural when it comes to generalizing Sturmian words. The reasons are explained in Section 2.

Our main result confirms Rauzy's conjecture.

\begin{theorem}\label{th:Rauzy_conj}
There does not exist any ternary word  with rationally independent letter frequencies and constant abelian complexity  equal to $3$.
\end{theorem}

The proof of Theorem~\ref{th:Rauzy_conj} is mostly algebraic. It relies on the introduction of a new renormalization process on infinite words, called \emph{abelian induction},  together with  Kronecker's diophantine theorem. 
Like the Rauzy--Veech induction for interval exchange transformations, this renormalization process admits a cocycle structure closely related to the representations of real numbers. The authors are  working on a second paper developing the theory of abelian induction. 

\medskip Theorem~\ref{th:Rauzy_conj}  is also true, and much easier to prove, when $d\geq 4$. Indeed, Currie and Rampersad   showed that having a constant abelian complexity is extremely constraining when 
$d\geq4$ \cite{CR11}. Building on their work, we  establish the following theorem.

\begin{theorem}\label{th:Rauzy_conj-dary}
Let $d\geq 4$. There does not exist any $d$-ary word  with rationally independent letter frequencies and constant abelian complexity  equal to $d$.
\end{theorem}

Using abelian induction, we also easily prove the following proposition.

\begin{proposition}\label{prop:belowbound_abcomp} Let $d \geq 1$ and $w$ be an infinite $d$-ary word with rationally independent letter frequencies. Then for all $n \in \N_{>0}$, we have 
$\ac_w(n)\geq d.$
\end{proposition}

From Theorems~\ref{th:Rauzy_conj}, \ref{th:Rauzy_conj-dary}, Proposition~\ref{prop:belowbound_abcomp}, and using again abelian induction, we can
show that in every infinite $d$-ary word with rationally independent letter frequencies, there exists no linear sequence of lengths for which the abelian complexity is at most $d$. In particular, the abelian complexity of such a word cannot even be \emph{eventually constant}, equal to $d$.

\begin{corollary}\label{cor:extraction_lineaire}
Let $d \geq 3$. Let $w$ be an infinite $d$-ary word that admits letter frequencies. If there exists $\ell\in\N_{>0}$ such that for every $n\in\N_{>0}$, $\ac_{w}(\ell n)\leq d$, then the letter frequencies of $w$ are rationally dependent.
\end{corollary}

Interestingly, Corollary~\ref{cor:extraction_lineaire} is optimal in the following sense: for every $d\geq 3$,
there exist infinite $d$-ary words with rationally independent letter frequencies  whose abelian complexity takes the value $d$ infinitely often; the $d$-bonacci word is one such example \cite{RSZ10,Tur13,Tur15}.

\begin{corollary}\label{coro:not-eventually-constant}Let $d \geq 3$. There exists no infinite $d$-ary word with rationally independent letter frequencies and eventually constant abelian complexity equal to $d$.
\end{corollary}

Note that Corollary~\ref{coro:not-eventually-constant} does not contradict the main theorem of Saarela \cite{Saa09}, which states that for all $d\geq 2$, there exists an infinite word whose abelian complexity is eventually constant, equal to $d$. Indeed, the words constructed by Saarela are binary.

\medskip
Finally, let us say a few words about the natural question behind the conjecture of Rauzy: what is the appropriate generalization of the Coven--Hedlund Theorem \ref{th:CH73}, that is, which class of ternary words with rationally independent letter frequencies enjoys a \emph{minimal}, or \emph{small}, or maybe even \emph{constant} abelian complexity?

At the time of writing, we are aware of two classes of infinite ternary words with rationally independent letter frequencies whose abelian complexity is less than or equal to $4$: these are the class of cubic billiard words generated by momenta with rationally independent coordinates \cite{AV23}, and the class of ``Sturmian-colored Sturmian words'' \cite{DMP24}. At the same time, we do not know any infinite ternary word with rationally independent letter frequencies whose abelian complexity is bounded above by $4$, and takes the value $3$ infinitely many times. (Note that the Tribonacci word mentioned above is not a counterexample since its abelian complexity also takes  each of the values $5,6$ and $7$ infinitely often \cite{RSZ10,Tur15}. The class of ternary words constructed by Kabor\'e and Tapsoba in \cite{KT07}, whose abelian complexity oscillates between $2$ and $4$, is not a counterexample either: these words have rationally dependent letter frequencies.)

\medskip
We conclude the introduction with two open questions.

\begin{question}
What is the appropriate generalization of Rauzy's conjecture when the alphabet size is $d \geq 4$? The question can be formalized as follows.
For $d \geq 1$, let $\kappa(d)$ denote the smallest integer for which there exists an infinite $d$-ary word with rationally independent letter frequencies and abelian complexity bounded above by $\kappa(d)$.

It is easy to check that $\kappa(1)=1$, $\kappa(2)=2$ (by the Coven--Hedlund Theorem~\ref{th:CH73}), and $\kappa(3)=4$ (by Theorem~\ref{th:Rauzy_conj} and the discussion above). Furthermore, it follows from Theorem~\ref{th:Rauzy_conj-dary} that, for every $d \geq 4$,
\[
\kappa(d) \geq d+1.
\]
This bound is probably of poor quality: to establish it, we barely use the rational independence of the letter frequencies.
In fact, keeping in mind that, in the cases $d=1,2,3$, the optimal bound is attained by $d$-ary hypercubic billiard words, whose abelian complexity is eventually constant and equal to $2^{d-1}$ \cite{AV23}, one may ask whether
\[
\kappa(d)=2^{d-1}
\]
for every $d \in \mathbb{N}_{>0}$.
\end{question}

\begin{question} For $d \geq 3$, can we dynamically or geometrically understand, or better, characterize the set of $d$-ary words with rationally independent letter frequencies and ``minimal''  abelian complexity? The question is already open for $d=3$.
\end{question}

\paragraph{Outline of the paper.}  In Section~\ref{sect:good_condition}, we explain why requiring the letter frequencies to be rationally independent is a natural condition when generalizing Sturmian words to larger alphabets. In Section~\ref{sect:proof_Rauzy_conj}, we prove our main result: Theorem~\ref{th:Rauzy_conj} (Rauzy's conjecture). In Section~\ref{sect:Other_proofs}, we prove Theorem~\ref{th:Rauzy_conj-dary}, and Corollaries~\ref{cor:extraction_lineaire} and \ref{coro:not-eventually-constant}. Proposition~\ref{prop:belowbound_abcomp} is proven in Section~\ref{sss:preliminaries_ab_induction}: it emerges as a consequence of the basic properties of abelian induction.

\bigskip

\emph{Some results of this paper have been announced, but not proved, in the conference long abstract \cite{AV23}.}


\section{Why the rational independence of the letter frequencies is a natural condition }\label{sect:good_condition}

In this section, we explain why the rational independence of the letter frequencies is a natural condition to require when one seeks to generalize Sturmian words.

\medskip

First, both Sturmian words and their standard combinatorial and dynamical generalizations have rationally independent letter frequencies. 
This is notably the case for Arnoux-Rauzy words (which include the Tribonacci word) \cite{And21,DHS22} and more generally $d$-ary strict episturmian words \cite[Chapter 4]{And21thesis}, Cassaigne-Selmer words \cite{CLL22}, but also the words encoding minimal trajectories in a hypercubic billiard table (see, for example, \cite{Bar95}). 

In fact, one may say that the rational independence of their letter frequencies is a core property of Sturmian words. Indeed, this property stems from their intimate connection with continued fractions (see, for instance, \cite[Chapter 6]{Pytheas02} or \cite[Chapter 1]{And26} in which this connection is carefully explained) and  manifests itself in all their classical dynamical and geometrical characterizations, for instance as the symbolic codings  of  \emph{irrational} rotations of the circle, \emph{minimal} square billiard trajectories, \emph{minimal} linear flows on the
two-dimensional torus, and as the digitization of lines in the plane with an \emph{irrational slope}.

\medskip
Secondly, both the converse implication of the Morse--Hedlund Theorem~\ref{th:MH38} (which is its interesting part)   and the Coven--Hedlund Theorem~\ref{th:CH73} admit equivalent formulations in terms of rational dependence of letter frequencies.

\begin{proposition} [Reformulation of the Morse--Hedlund Theorem~\ref{th:MH38}] \label{prop:MH38bis}  Let $w$ be an infinite binary word. If there exists $n \in \N$ such that $w$ admits at most $n$ subwords of length $n$, then the letter frequencies of $w$ exist and are rationally dependent. 
\end{proposition}

\begin{proposition} [Reformulation of the Coven--Hedlund Theorem~\ref{th:CH73}] \label{prop:CH73bis} An infinite word is Sturmian if
and only if its abelian complexity is constant, equal to 2, and if its letter
frequencies (which exist) are rationally independent.
\end{proposition}

(To derive Propositions~\ref{prop:MH38bis} and \ref{prop:CH73bis} from Theorems~\ref{th:MH38} and \ref{th:CH73}, we simply use the fact that an eventually periodic word  has rational letter frequencies.)
Furthermore, it was recently rediscovered by the first author and Cassaigne   that with this new formulation, the Morse--Hedlund theorem can be generalized to larger alphabets.

\begin{theorem} [Tijdeman, 1991, \cite{Tij99}, see also \cite{And26}(Chapter 3) for an algebraic proof] \label{th:tijdeman}  Let $d\geq 1$. Let $w$ be an infinite $d$-ary word that admits letter frequencies. If there exists $n \in \N$ such that $w$ has at most $(d-1)n$ subwords of length $n$, then the letter frequencies of $w$  are rationally dependent. 
\end{theorem}

For all the reasons explained above, the authors believe that the rational independence of the letter frequencies is the natural condition in which Rauzy's Conjecture~\ref{conj:Rauzy} should be considered.

\medskip
Finally, it is worth mentioning that this condition eliminates the trivial counterexamples \eqref{eq:counterexamples} given in the Introduction, as well as the two remarkable families of counterexamples constructed by Richomme, Saari and Zamboni in 2011.

\begin{example}\label{ex:RSZ}
We justify that the two families of ternary words constructed by Richomme, Saari and Zamboni \cite{RSZ11} have rationally dependent (and in fact highly structured) letter frequencies. These families are defined as follows:  
\begin{itemize}
\item $\Fcal_1 = \Big\{\sigma (w) \text{ s.t. } w \in \{0,1\}^{\N} \text{ is uniformly recurrent and not eventually periodic}\Big\}$,\\ where $\sigma$ is the substitution 
\[\begin{array}{llll}\sigma : &0 &\mapsto &123,\\
& 1 & \mapsto & 132.
\end{array}\]
\item $\Fcal_2 = \Big\{w \in \{1,2,3\}^{\N} : \begin{array}{c}\text{$w$ is uniformly recurrent, $1$-balanced,}\\ \text{but not eventually periodic}\end{array} \Big\}$.
\end{itemize}

It follows from \cite{Gra73,Hub00} that a word $w \in \Fcal_2$ if and only if there exists a Sturmian word $w_0 \in \{0,1\}^{\N}$ such that (up to exchanging the roles of  $1$, $2$ and $3$) \[w=\mathrm{color}(w_0),\] where $\mathrm{color}$ is the map that transforms the $n$-th occurrence of the letter $0$ in $w_0$ by
\[\begin{cases}
2 \text{ if $n$ is even}, \\
3 \text{ otherwise.}
\end{cases}\]

All words in $\Fcal_1$ and $\Fcal_2$ have rationally dependent  letter frequencies. Indeed, for every $w \in \Fcal_1$, the vector of letter frequencies is $\bm{f_w}=(1/3,1/3,1/3)$, while for every $w \in \Fcal_2$,  $\bm{f_w}$ lies in the hyperplane $f_w(2)=f_w(3)$. It is also worth noting  that $\Fcal_1$ and $\Fcal_2$ do not form satisfactory ternary generalizations of Sturmian words (in contrast with, for example, Arnoux-Rauzy, natural codings of intervals exchange transformations, or cubic billiards words). 
\end{example}

\section{Proof of Theorem~\ref{th:Rauzy_conj} (Rauzy's conjecture)\label{sect:proof_Rauzy_conj}}

In this section, we prove that there is no ternary word with rationally independent letter frequencies and constant abelian complexity.
The proof relies on a renormalization process for infinite words, which we call \emph{abelian induction}, together with Kronecker's  theorem. The authors are currently working on a second paper developing the abelian induction.

\subsection{Structure of the proof of Theorem~\ref{th:Rauzy_conj}}

\medskip
Let $w \in \Acal^{\N}$ be an infinite word written over an alphabet $\Acal$. We consider the following three properties:

\medskip
\begin{tabular}{rl}
    \Hcomp & The abelian complexity of $w$ is constant. (Observe that in this case, the\\
    & constant is necessarily the alphabet size: $ \card \Acal = \ac_w(1) $.)\\
    \Hfreq & The letter frequencies of $w$ are rationally independent.\\
    \Hbal & The word $w$ is $1$-balanced on at least two letters.
\end{tabular}
\medskip

We recall that $w$ is said \emph{$C$-balanced} (with $C \in \N$) \emph{on a letter} $i \in \Acal$ if for every pair of equally long subwords $u$ and $v$ in $w$, the number of occurrences of the letter $i$ in $u$ and $v$ differs by at most $C$:
\[\Big\vert \vert u \vert_i - \vert v \vert_i \Big\vert \leq C.\] A word $w$ is  \emph{$C$-balanced} if it is $C$-balanced on all its letters. It is easy to check (see also \cite{RSZ11}) that the balance constant $C$ of an infinite word $w$ and its abelian complexity $\ac_w$ are related. More precisely, for every  $n \in\N_{>0}$, we have
\begin{equation}\label{eq:balanceabcomplexity}
    C+1 \leq \ac_w(n) \leq (C+1)^{\card \Acal-1}.
\end{equation}
(The first inequality follows from a discrete intermediate value argument.)

\medskip
To prove Theorem~\ref{th:Rauzy_conj}, we argue by contradiction and assume that there exists a word $w\in\{1,2,3\}^\N$ satisfying \Hcomp and \Hfreq. We derive a contradiction in two steps.
\begin{itemize}
    \item \emph{Step 1: abelian induction.} We prove that
     the existence of a ternary word $w$ satisfying \Hcomp and \Hfreq implies the existence a ternary word $w'$ also satisfying \Hcomp and \Hfreq, and satisfying furthermore the hypothesis \Hbal.

In order to do so, we study the \emph{abelian induction} over words with constant abelian complexity equal to $3$. Namely, for every length $\ell\in\N_{>0}$ we consider the new word $I_\ell(w)$ obtained from $w$ as follows. Denote by $\alpha$, $\beta$ and $\gamma$ the three abelian classes of equivalence of length-$\ell$ subwords of $w$:
\[
    \lan_\ell(w)/_{\sim_{ab}}=\{\alpha,\beta,\gamma\}.
\]
Then, the word $I_\ell(w)$ is the right infinite word written on the new ternary alphabet $\{\alpha,\beta,\gamma\}$ whose $n$-th letter, for every $n \in \N$, is given by: 
\[I_\ell(w)[n]=\overline{w[n\ell:(n+1)\ell-1]}.\]
In this expression, $w[n\ell:(n+1)\ell-1]$ denotes the length-$\ell$ subword of $w$ spanning from  position $n\ell$  to position $(n+1)\ell-1$ (both included) and $\overline{w[n\ell:(n+1)\ell-1]}$ its abelian class of equivalence. (Examples and properties of the abelian induction will be given in Section~\ref{sss:preliminaries_ab_induction}.)

We prove that:
\begin{enumerate}
    \item If $w$ is a ternary word satisfying \Hcomp and \Hfreq, then for every length $\ell\in\N_{>0}$, $I_\ell(w)$ is also a ternary word satisfying \Hcomp and \Hfreq; see Proposition~\ref{prop:stabilite_induction}.
    \item If $w$ is a ternary word satisfying \Hcomp and \Hfreq, then, there exists $\ell\in\N_{>0}$ such that $I_\ell(w)$  satisfies furthermore the third condition \Hbal; see Proposition~\ref{prop:induction_error}.
\end{enumerate}
    \item \emph{Step 2: the contradiction.} We prove that no ternary word can satisfy the three conditions \Hcomp, \Hfreq, and \Hbal simultaneously; see Proposition~\ref{prop:last_step}. Indeed, assume that such a word $w' \in \{1,2,3\}^{\N}$ exists, and that $1$ and $2$ are the two $1$-balanced letters of $w'$. Denote by $\sigma_1$ and $\sigma_2$ the substitutions defined by
\[
    \begin{array}{rlrl}
        \sigma_1: & 1\mapsto 1 &\hspace{2cm}\sigma_{2}: & 1\mapsto 0\\
        & 2\mapsto 0 && 2\mapsto 2\\
        & 3\mapsto 0 && 3\mapsto 0.
    \end{array}
\]
We prove that $\sigma_1(w')$ and $\sigma_2(w')$ are Sturmian words.  We then exploit this fact  to justify that there exists $n\in\N$ such that $\sigma_1(w')[n]=1$ and $\sigma_2(w')[n]=2$. This will be our contradiction: indeed, $\sigma_1(w')[n]=1$ implies $w'[n]=1$, while $\sigma_2(w')[n]=2$ implies $w'[n]=2$. At this point, the proof of Theorem~\ref{th:Rauzy_conj} (Rauzy's conjecture) will be complete.
\end{itemize}


\subsection{Step 1: Abelian induction}\label{ss:step1}

The aim of this subsection is to prove the following two propositions.

\begin{proposition}\label{prop:stabilite_induction}
Let $d \geq 1$. If  $w$ is an infinite $d$-ary word satisfying \Hcomp and \Hfreq, then for every $\ell\in\N_{>0}$, the induced word $I_\ell(w)$ is also a $d$-ary word satisfying \Hcomp and \Hfreq.
\end{proposition}

\begin{proposition}\label{prop:induction_error}
If  $w$ is an infinite ternary word satisfying \Hcomp and \Hfreq, then there exists $\ell\in\N_{>0}$ such that the induced word $I_\ell(w)$ is a ternary word satisfying the three conditions \Hcomp, \Hfreq and \Hbal.
\end{proposition}

\subsubsection{Abelian induction: definition, examples and first properties}\label{sss:preliminaries_ab_induction}

The \emph{abelian induced word} of any infinite word $w \in\Acal^{\N}$ for a given length $\ell \in \N_{>0}$ is the word $w':=I_\ell(w)$ written on the new alphabet $\mathcal{A}_\ell:=\lan_\ell(w)/_{\sim_{ab}}$ (which depends on $w$) and defined as follows: for every $n\in\N$, 
\[
    I_\ell(w)[n]:=\overline{w[n\ell:(n+1)\ell-1]},
\]
where $\overline{u}$ denotes the abelian equivalence class of the finite word $u$. 
For example, for the binary word $w= 122112122112...$, we have
\[\begin{array}{ll}
I_2(w) =  aaaaaa..., \qquad & \text {with } a = \overline{12} = \{12,21\};\\
I_3(w) = BABA..., & \text {with } A= \overline{112} = \{112,121,211\}, B=\overline{122} = \{122,212,221\}.\\
\end{array}
\]
Note that we can extend this definition to the set of finite words $u$ whose length is a multiple of $\ell$. The induced word $I_\ell(u)$ is then a  word of length  $\vert u \vert /\ell$.

\medskip
We recall that the \emph{abelian vector} (sometimes also called \emph{Parikh vector} or \emph{population vector} in the literature) of a finite word $u$ written over an alphabet $\Acal$ is the column vector that counts the number of occurrences of each letter in $u$: 
\[\ab(u):=(|u|_a)_{a\in\Acal}.\]
(For example, we have $\ab(1121)= \,^t(3,1)$.) Clearly, abelian vectors form a complete set of invariants for the equivalence relation $\sim_{ab}$.
For $w \in \Acal^{\N}$ and $\ell \in \N_{>0}$, we can thus define the rectangular matrix $\bm{M}_\ell(w)$ whose columns are the abelian vectors representing the equivalence classes in $\mathcal{A}_\ell = \lan_\ell(w)/_{\sim_{ab}}$:
\[\bm{M}_\ell(w) = \Big(\ab(\alpha)\Big)_{\alpha \in \mathcal{A}_\ell}.\]
The size of the matrix $\bm{M}_\ell(w)$ is  $(\card \Acal) \times (\card \Acal_\ell)$.

\medskip

The next lemma states some basic properties of $\bm{M}_\ell(w)$ and $I_\ell(w)$. They will be used to prove Proposition~\ref{prop:belowbound_abcomp} (which is proved at the end of this subsubsection), but also Propositions \ref{prop:stabilite_induction} and \ref{prop:induction_error}, Theorem~\ref{th:Rauzy_conj-dary}, and Corollary~\ref{cor:extraction_lineaire}.

\begin{lemma}\label{lemma:induction_basics} Let $\Acal$ be an alphabet, $w \in \Acal^{\N}$ an infinite word admitting letter frequencies, and $\ell \in \N_{>0}$.
\begin{enumerate}
\item[\emph{1.}] For every $n\in\N_{>0}$ and $u\in\lan_{n\ell}(w)$, we have $\ab(u)=\bm{M}_\ell(w)\cdot\ab(I_\ell(u))$.
\item[\emph{2.}] If the letter frequencies of $I_\ell(w)$ exist, they satisfy
\[
    (f_w(a))_{a \in \Acal} = \frac{1}{\ell} \bm{M}_\ell(w)\cdot  (f_{I_\ell(w)}(\alpha))_{\alpha\in \Acal_\ell}.
\]
If not, the relation still holds by replacing the vector of letter frequencies of $I_\ell(w)$ by any $\bm{f'}=(f'(\alpha))_{\alpha \in \Acal_{\ell}} \in F$, where $F$ is the set of all subsequential limits of the sequence
\[ \Bigg(\ab\Big(\frac{\pref_n(I_\ell(w))}{n}\Big)\Bigg)_n.\]

\item[\emph{3.}] We have
\[\dim \Span_{\Q} (f_w(a): a \in \Acal) \leq \min_{\bm{f'} \in F} \dim  \Span_{\Q} (f'(\alpha) : \alpha \in \Acal_\ell).\]
\end{enumerate}
\end{lemma}

\begin{remark}
In Lemma~\ref{lemma:induction_basics}, Assertion 1, the vectors $\ab(u)$ and $\ab(I_\ell(u))$ are of different nature (they may even have different dimension). Indeed, $\ab(u)$ counts the number of occurrences of the letters in $\Acal$, while $\ab(I_\ell(u))$ counts the number of occurrences of the letters in the induced alphabet $\Acal_\ell$.
\end{remark}

\begin{proof}
Denote by $u_1,\ldots,u_n$ the $\ell$-length words over $\Acal$ such that $u = u_1\cdots u_n$. We have

\[\vert u \vert_a = \sum_{i = 1}^n \vert u_i \vert_a = \sum_{i = 1}^n \vert \overline{u_i} \vert_a = \sum_{\alpha \in\Acal_\ell} \vert I_l(u)\vert_\alpha \cdot \vert \alpha \vert_a\]
from which the desired matricial expression of Assertion $1$ follows.

\bigskip

\noindent
We now establish Assertion $2$. Let $\bm{f'} \in F$, and $(n_k)_{k \in \N}$ such that 
\[ \frac{\ab(\pref_{n_k}(I_\ell(w)))}{n_k} \underset{k\to\infty}{\longrightarrow} \bm{f'}.\]
By the definition of abelian induction on finite words, we have 
\[\pref_{n_k}(I_\ell(w)) = I_\ell(\pref_{\ell n_k}(w))\] for every ${k \in \N}$. Applying Assertion 1 to $u:=\pref_{\ell n_k}(w)$, we 
obtain
\begin{equation}\label{eqloc:prefrequence}
    \ab(\pref_{\ell n_k}(w)) =  \bm{M}_\ell(w)\cdot \ab(\pref_{n_k }(I_\ell(w))).\end{equation}
Dividing both sides of Equation \eqref{eqloc:prefrequence} by $\ell n_k$, and letting $k \to \infty$, we obtain the desired expression
\begin{equation}\label{eq:assertion2}
    \bm{f_w} = \frac{1}{\ell} \bm{M}_\ell(w)\cdot  \bm{f'},
\end{equation}
where $\bm{f_w}$ is the vector of letter frequencies of $w$, which was assumed to exist.

\bigskip

\noindent
Finally, the proof of Assertion $3$  follows immediately from the relation \eqref{eq:assertion2}, noticing that the matrix $\bm{M}_\ell(w)$ has integer entries.  
\end{proof}

We are already in a position to prove Proposition~\ref{prop:belowbound_abcomp}, announced in the Introduction.

\begin{proof}[Proof of Proposition~\ref{prop:belowbound_abcomp}] The case $d=1$ being trivial, let $d\geq 2$. We proceed by contraposition. If there exists $n \in \N_{>0}$ such that $\ac_w(n) \leq d-1$, then the induced word $I_n(w)$ is written with at most $d-1$ letters. By Lemma~\ref{lemma:induction_basics} Assertion $3$, we thus have
\[\dim \Span_{\Q} (f_w(a): a \in \{1,\ldots,d\}) \leq d-1,\]meaning that $w$ has rational dependent letter frequencies.
\end{proof}


\subsubsection{Proof of Proposition \ref{prop:stabilite_induction}}
The case $d=1$ being trivial, let $d\geq 2$. Let $w \in \{1,\ldots,d\}^{\N}$ be a $d$-ary word satisfying both conditions \Hcomp and \Hfreq, and let $\ell\in\N_{>0}$. We want to prove that the induced word $I_\ell(w)$ is also a $d$-ary word satisfying   \Hcomp and \Hfreq. The proof is made of three lemmas.

\begin{lemma}\label{lemma:GL3} The induced word $I_\ell(w)$ is a $d$-ary word, and  $\bm{M}_\ell(w)\in\GL_d(\Q)$.
\end{lemma}

\begin{proof}First, it immediately follows from \Hcomp that $\ac_\ell(w)=d$. Therefore, the induced alphabet $\Acal_\ell$ is  made of $d$ letters, and $\bm{M}_\ell(w)$ is a $d\times d$ square matrix. 
We now prove that $\ker \,^t\bm{M}_\ell=\{0\}$. (Once this is proven, we conclude that $^t\bm{M}_\ell$ and $\bm{M}_\ell$ belong to $\GL_d(\Q)$.) Keeping the notations of Lemma~\ref{lemma:induction_basics}, we denote by $\bm{f_w}$ the column vector formed by the letter frequencies of $w$ (which exist by \Hcomp), and by $\bm{f'} \in F$ any subsequential limit of the sequence 
\[ \Bigg(\ab\Big(\frac{\pref_n(I_\ell(w))}{n}\Big)\Bigg)_n\]
in the compact set $[0,1]^d$.
Let $\bm{q}\in \Q^{d}$ be in the kernel of $^t\bm{M}_\ell(w)$, and denote by $\langle \bm{q} | \bm{f_w}\rangle$ the inner product of $\bm{q}$ and $\bm{f_w}$. We have \[
    \langle \bm{q} | \bm{f_w}\rangle = \Big\langle \bm{q} \; \Big|\; \frac{1}{\ell}\bm{M}_\ell(w)\cdot \bm{f'}\Big\rangle = \frac{1}{\ell}\big\langle{}^t\bm{M}_\ell(w)\cdot \bm{q} \; \big|\; \bm{f'}\big\rangle = 0
\]
(in which the first equality stems from Lemma \ref{lemma:induction_basics}, Assertion $2$). It follows from the calculation above and the rational independence of the letter frequencies of $w$ that $\bm{q}=0$. The proof of the lemma is complete.
\end{proof}

\begin{lemma}\label{lemma:induced_frequencies_exist}
The letter frequencies of $I_\ell(w)$ exist and are rationally independent.
\end{lemma}

\begin{proof}
    Since the matrix $\bm{M}_\ell(w)$ is invertible by Lemma~\ref{lemma:GL3},  Assertion $2$ of Lemma~\ref{lemma:induction_basics} is equivalent to: 
\[
    \bm{f'}=\ell \bm{M}_\ell(w)^{-1}\cdot \bm{f_w}
\]
for every $\bm{f'} \in F$. This ensures that there is only one possible value for $\bm{f'}$--- in other words, that the induced word $I_l(w)$ admits letter frequencies. 
 Moreover, by Lemma~\ref{lemma:induction_basics} Assertion $3$, we have
\[
     \dim\Span_{\Q} (f_w(i):i \in \{1,\ldots,d\})\leq \dim \Span_\Q(f_{I_\ell(w)}(\alpha):\alpha \in \Acal_\ell).
\]
Since the left member is equal to $d$ by \Hfreq, and since the right member is bounded above by $\card \Acal_\ell = d$, we conclude that $\dim \Span_\Q(f_{I_\ell(w)}(\alpha):\alpha \in \Acal_\ell)=d$. The proof of the lemma is complete.
\end{proof}

\begin{lemma}\label{lemma:induced_abconstant}
    The abelian complexity of the induced word $I_\ell(w)$ is constant: for every $n\in\N_{>0}$, we have 
    \[\ac_{I_\ell(w)}(n)= d.\]
\end{lemma}

\begin{proof}
    Let $n\in \N_{>0}$. It follows from Proposition~\ref{prop:belowbound_abcomp} and Lemma~\ref{lemma:induced_frequencies_exist} that 
    $\ac_{I_\ell(w)}(n)\geq d$. It thus remains to prove that $\ac_{I_\ell(w)}(n)\leq d$. To do this, we justify that two non abelian-equivalent subwords in $\lan_n(I_\ell(w))$ always come from two non abelian-equivalent subwords in $\lan_{n\ell}(w)$. Indeed, by Lemma~\ref{lemma:induction_basics} Assertion 1, we have  
\[
    \ab(u')-\ab(v')=\bm{M}_\ell(w)^{-1}\cdot\big[\ab(u)-\ab(v)\big]
\]
for every $u',v'\in \lan_n(I_\ell(w))$ and every $u,v\in\lan_{n\ell}(w)$ such that $u'=I_\ell(u)$ and $v' = I_\ell(v)$. This expression shows that if $\ab(u')-\ab(v') \neq 0$, then  $\ab(u)-\ab(v)\neq 0$, hence
\[
    \ac_{I_\ell(w)}(n) = \card(\lan_{n}(I_\ell(w))/_{\sim_{ab}}) \leq \card(\lan_{n\ell}(w)/_{\sim_{ab}}) = \ac_w(n\ell) = d.
\]
We thus have $\ac_{I_\ell(w)}(n)\leq d$. The proof of the lemma is complete.
\end{proof}

It follows from Lemmas~\ref{lemma:GL3}, \ref{lemma:induced_frequencies_exist} and \ref{lemma:induced_abconstant} that $I_\ell(w)$ is a $d$-ary word that satisfies \Hcomp and \Hfreq. The proof of Proposition~\ref{prop:stabilite_induction} is complete.


\subsubsection{Proof of Proposition \ref{prop:induction_error}}

Let $w \in \{1,2,3\}^{\N}$ be an infinite ternary word satisfying \Hcomp and \Hfreq. We  want to prove that there exists $\ell \in \N_{>0}$ such that  $I_\ell(w)$, which is already known to satisfy \Hcomp and \Hfreq by Proposition~\ref{prop:stabilite_induction}, also satisfies \Hbal. 

First, observe that if  $w$ is $1$-balanced, there is nothing to prove: the word $I_1(w)=w$ trivially satisfies \Hbal. In the remainder of the proof, we assume that $w$ is not $1$-balanced. Since $w$ is, however, $2$-balanced by \Hcomp and Equation~\eqref{eq:balanceabcomplexity}, there exist $\ell\in\N_{>0}$, $u,v\in\lan_\ell(w)$, and $i\in\{1,2,3\}$ such that 
\begin{equation}\label{eq:2balanced}
|u|_i-|v|_i=2.\end{equation}  
Without loss of generality, we assume that $i=1$.  We  prove that the induced word $w':=I_\ell(w)$ satisfies \Hbal.

\medskip
We begin by describing the abelian vectors of subwords of length $\ell$ in $w$.

\begin{lemma}\label{lemma:language_ijk}
Up to permuting the role of the letters $2$ and $3$ in $w$, there exists $r,s,t\in\N$ such that
\begin{equation}\label{eq:abelian_language}
    \ab\big(\lan_\ell(w)\big) = \left\{\begin{pmatrix} r \\s \\t \end{pmatrix}, \begin{pmatrix} r+1 \\s-1 \\t \end{pmatrix}, \begin{pmatrix} r+2 \\s-1 \\t-1 \end{pmatrix}\right\}.
\end{equation}
\end{lemma}

\begin{proof}
It suffices to examine all the triplets of vectors of equal sum whose set of top coordinates respects Equation~\eqref{eq:2balanced},
 and whose $i$-th coordinates form a set of consecutive integers, for all $i \in \{1,2,3\}$. (This second condition is imposed by the fact that when a window of size $\ell$ is slid along $w$, the number of occurrences of the letter $i$ in the window changes by at most $1$ at each step.) 

Up to exchanging the roles of the letters $2$ and $3$, this leaves two possibilities: the situation \eqref{eq:abelian_language}, and 
\[
    \ab\big(\lan_\ell(w)\big) = \left\{\begin{pmatrix} r \\s \\t \end{pmatrix}, \begin{pmatrix} r+1 \\s-1 \\t \end{pmatrix}, \begin{pmatrix} r+2 \\s-2 \\t \end{pmatrix}\right\}.
\]
But in this second case, the induction matrix $\bm{M}_\ell(w)$ does not belong to $\GL_3(\Q)$, which contradicts Lemma~\ref{lemma:GL3}. Therefore, there is no other choice than having \eqref{eq:abelian_language}.
\end{proof}

In the sequel, we identify the  letters of the induced alphabet $\Acal_\ell=\lan_\ell(w)/_{\sim_{ab}}$ with their abelian vectors in $\ab(\lan_\ell(w))$.

\begin{lemma}\label{lemma:Hbalforinducedword}The letters $\begin{pmatrix} r \\s \\t \end{pmatrix}$ and $\begin{pmatrix} r+2 \\s-2 \\t \end{pmatrix}$ are $1$-balanced in $w'$.
\end{lemma}

Once Lemma~\ref{lemma:Hbalforinducedword} is proven, the proof of Proposition~\ref{prop:induction_error} is complete: we have shown that $w'=I_\ell(w)$ satisfies \Hbal.

\begin{proof}[Proof of Lemma \ref{lemma:Hbalforinducedword}] Assume that  there exist $n\in\N_{>0}$, $u',v'\in\lan_n(w')$, and $\alpha\in\Acal_\ell$ such that $|u'|_\alpha-|v'|_\alpha\geq2$. Denote by $\beta$ and $\gamma$ the two other letters in $\Acal_\ell$.  We prove that $\alpha = {}^t(r+1,s-1,t)$. First, since $w'$ satisfies \Hcomp and \Hfreq, the same arguments as those used to prove Lemma~\ref{lemma:language_ijk} show that, up to exchanging the roles of the letters $\beta$ and $\gamma$ in $w'$, there exists $r',s',t'\in\N$ such that
\[
    \ab\big(\lan_n(w')\big) = \left\{\begin{pmatrix} r' \\s' \\t' \end{pmatrix}, \begin{pmatrix} r'+1 \\s'-1 \\t' \end{pmatrix}, \begin{pmatrix} r'+2 \\s'-1 \\t'-1 \end{pmatrix}\right\},
\]
the successive rows of these vectors being indexed by $\alpha$, $\beta$ and $\gamma$.
Hence
\[\ab(u')=\begin{pmatrix} |u'|_\alpha \\|u'|_\beta \\|u'|_\gamma \end{pmatrix} = \begin{pmatrix} r'+2 \\s'-1 \\t'-1 \end{pmatrix}
     \quad\text{ and }\quad 
     \ab(v')=\begin{pmatrix} |v'|_\alpha \\|v'|_\beta \\|v'|_\gamma \end{pmatrix} = \begin{pmatrix} r' \\s' \\t' \end{pmatrix}.  
\]
Let $x,y\in\lan_{n\ell}(w)$ be such that $I_\ell(x)=u'$ and $I_\ell(y)=v'$. Then we have:\
\[
    \begin{array}{rcl}
        |y|_1-|x|_1 &=& \sum\limits_{\delta\in \Acal_\ell}\big(|v'|_\delta-|u'|_{\delta}\big)|\delta|_1\\
        \\
        &=&2|\alpha|_1-|\beta|_1-|\gamma|_1.
    \end{array}
\]
The last equality implies $\alpha = {}^t(r+1,s-1,t)$. Indeed, if $\alpha \neq {}^t(r+1,s-1,t)$, then we have either $\alpha = {}^t(r,s,t)$ or $\alpha = {}^t(r+2,s-1,t-1)$ by Lemma~\ref{lemma:language_ijk}. The first case yields
\[
    |y|_1-|x|_1 = 2r-(r+1)-(r+2)=-3.
\]
A similar calculation shows that $|y|_1-|x|_1= 3$ in the second case. Both cases contradict the $2$-balancedness of $w$. 

The proofs of Lemma \ref{lemma:Hbalforinducedword} and Proposition \ref{prop:induction_error} are complete.\end{proof}

\subsection{Step 2: The contradiction}\label{ss:step3}

The aim of this section is to prove the following proposition.

\begin{proposition}\label{prop:last_step}
    Let $d\geq 3$. There  exists no $d$-ary word satisfying \Hfreq and \Hbal.
\end{proposition} 

\begin{remark}In the sequel, we will only use Proposition~\ref{prop:last_step} for $d=3$.
\end{remark}

\begin{remark}A similar proof strategy was employed in \cite{Gra73} and \cite{Hub00} to study $d$-ary $1$-balanced words.
\end{remark}

\begin{proof} Let $d\geq 3$. We proceed by contradiction. Let $w \in\{1,\ldots,d\}^{\N}$ be an infinite $d$-ary word with rationally independent letter frequencies, and admitting two $1$-balanced letters, say $1$ and $2$. Denote by $w_1 = \sigma_1(w)$ and $w_2 = \sigma_2(w)$ the infinite binary words obtained as the images of $w$ by the substitutions $\sigma_1$ and $\sigma_2$ defined by:
\[
    \begin{array}{llll}
        \sigma_i: & i & \mapsto & i \\
        & j & \mapsto & 0 \quad \text{ for all } j\in\{1,\ldots,d\}\backslash\{i\}.
        
    \end{array}
\]

\begin{lemma}\label{lemma:decolored_freq}
The  words $w_1$ and $w_2$ are Sturmian and their letter frequencies are respectively given by
\[
    \bm{f_{w_1}}:=\begin{pmatrix} f_{w_1}(0) \\ f_{w_1}(1) \end{pmatrix}=\begin{pmatrix}\sum_{j\neq 1}  f_w(j)  \\ f_w(1)\end{pmatrix} \quad\text{ and }\quad \bm{f_{w_2}}:=\begin{pmatrix} f_{w_2}(0) \\ f_{w_2}(2) \end{pmatrix}=\begin{pmatrix}\sum_{j\neq 2}  f_w(j) \\ f_w(2)\end{pmatrix}.
\]
\end{lemma}

\begin{proof}
First, it is clear that the letter frequencies of $w_1$, $w_2$ and $w$ satisfy the relations stated in Lemma~\ref{lemma:decolored_freq}. These relations further induce that $w_1$ and $w_2$ also have rationally independent letter frequencies. Therefore, $w_1$ and $w_2$ are not eventually periodic.
Finally, it follows from the $1$-balancedness of $w$ on the letters $1$ and $2$, that the binary words $w_1$ and $w_2$ are $1$-balanced.
By a classical theorem of Morse and Hedlund (\cite{MH40} or \cite[Chapter 2]{Loth}),  $w_1$ and $w_2$ are thus Sturmian.
\end{proof}

Since $w_1$ is Sturmian, there exist $x \in[0,1)$ and $\alpha\in[0,1)\backslash\Q$ such that $w_1$ encodes the trajectory of $x$ under the iteration of the rotation of angle $\alpha$ on the circle:
\[
    \begin{array}{rccl}
        R_\alpha\colon & \R/\Z & \longrightarrow & \R/\Z\\
        & z & \longmapsto & z+\alpha\bmod 1
    \end{array}
\]
with respect to one of the two partitions $[0,1-\alpha)\cup[1-\alpha,1)$ or $(0,1-\alpha]\cup(1-\alpha,1]$ of the circle (see, again, \cite{MH40} or \cite[Chapter 2]{Loth}). (To fix ideas, and since it has no importance in the sequel, we will work with the first partition.)
 More precisely, for every $n \in \N$, we have:
\[
    w_1[n]=0\;\iff\;R_\alpha^n(x)\in[0,1-\alpha),
\] 
and we know that $f_{w_1}(0)=1-\alpha$. 
Similarly, there exist $y \in[0,1)$ and $\beta\in[0,1)\backslash\Q$ such that  for every $n\in\N$, we have
\[
    w_2[n]=0\;\iff\;R_\beta^n(y)\in[0,1-\beta)
\]
and $f_{w_2}(0)=1-\beta$.
Now, let $R_{(\alpha,\beta)}$ be the rotation of angle $(\alpha,\beta)$ defined on the two dimensional torus $\R^2/\Z^2$ by
\[
    \begin{array}{rccl}
        R_{(\alpha,\beta)}\colon &\R^2/\Z^2&\longrightarrow&\R^2/\Z^2\\
        &(z,t)&\longmapsto&(z+\alpha,t+\beta)\bmod \Z^2=(R_\alpha(z),R_\beta(t)).
    \end{array}
\]

\begin{lemma}
The trajectory of every point in $\R^2/\Z^2$ under the iteration of the rotation $R_{(\alpha,\beta)}$ is dense in $\R^2/\Z^2$.
\end{lemma}

\begin{proof}By Kronecker's theorem, it suffices to prove that the three real numbers $1$, $\alpha$, and $\beta$ are rationally independent.
Let $q_1,q_2,q_3 \in \Q$ such that $q_1+q_2\alpha+q_3  \beta=0$. We have
\[ \begin{array}{lll}  0 = q_1+q_2\alpha+q_3  \beta
&=&
q_1+q_2f_{w_1}(1)+q_3  f_{w_2}(2)\\
& = & q_1\Big(\sum_{j=1}^{d}f_w(j)\Big) +q_2f_{w}(1)+q_3 f_{w}(2)\\
 & = &\sum_{j \in \{1,\ldots,d\}\backslash\{1,2\}}q_1f_w(j)+(q_1+q_2) f_w(1)+(q_1+q_3) f_w(2).
\end{array}\]
It follows from the rational independence of $f_w(1), \ldots, f_w(d)$ that $(q_1, q_1+q_2, q_1+q_3) = (0,0,0)$, hence $(q_1,q_2,q_3)=(0,0,0)$. The real numbers $1$, $\alpha$, and $\beta$ are therefore rationally independent.
\end{proof}

In particular, the trajectory of $(x,y)$ under the iteration of $R_{(\alpha,\beta)}$ is dense in $\R^2/\Z^2=[0,1)^2$. Therefore, there exists $n\in\N_{>0}$ such that $R_{(\alpha,\beta)}^n(x,y)=(R_\alpha^n(x),R_\beta^n(y))\in[1-\alpha,1)\times[1-\beta,1)$. For this particular value of $n$, we have $w_1[n]=1$ and $w_2[n]=2$, hence
$w[n]=1$ and $w[n]=2$. This is impossible.
The proof of Proposition \ref{prop:last_step} is complete.
\end{proof}

\subsection{Conclusion of the proof of Theorem~\ref{th:Rauzy_conj}}\label{ss:conclusion}

By Proposition~\ref{prop:induction_error}, if there exists an infinite ternary word satisfying \Hcomp and \Hfreq, then there also exists an infinite ternary word satisfying \Hcomp, \Hfreq and \Hbal. But the conditions \Hfreq and \Hbal are incompatible by Proposition~\ref{prop:last_step}. The proof of Theorem~\ref{th:Rauzy_conj} (Rauzy's conjecture) is complete.

\section{Proofs of Theorem~\ref{th:Rauzy_conj-dary}, Corollary~\ref{cor:extraction_lineaire} and Corollary~\ref{coro:not-eventually-constant}} \label{sect:Other_proofs}

\subsection{Proof of Theorem~\ref{th:Rauzy_conj-dary}}\label{ss:proof_Rauzy_dary}

Let $d \geq 4$, and let $w \in  \{1,\ldots,d\}^{\N}$ be an infinite $d$-ary word with constant abelian complexity. We prove that the letter frequencies of $w$ are rationally dependent. 
The proof is much simpler than in the case $d=3$. Indeed, the situation is so constrained that Currie and Rampersad already showed, with a combinatorial argument, that  $w$ cannot be recurrent \cite{CR11}. Reading carefully their proof, we can improve their result.

\begin{proposition}\label{prop:CR_improved}
    Let $d\geq 4$. No infinite $d$-ary word with constant abelian complexity  is abelian recurrent on its subwords of length $2$. 
\end{proposition}

We recall that an infinite word is \emph{recurrent} if each of its subwords appears infinitely often. It is \emph{abelian recurrent on its subwords of length $2$} if, for every $u \in \lan_2(w)$, the abelian equivalence class $\overline{u}$ appears infinitely often in $w$. This second property is weaker. For instance, the eventually periodic word 
$ w = 123\cdot1\cdot321321321321321321...$
---the dots are inserted for readability--- is not recurrent (the subword $12$ appears exactly once); but it
is abelian recurrent on its subwords of length $2$: indeed, the only abelian classes  to appear are $\overline{12}$, $\overline{23}$, $\overline{13}$, and they appear infinitely often.

\begin{proof}[Proof of Proposition~\ref{prop:CR_improved}]It suffices to read the main proof of \cite{CR11} and verify that whenever the recurrence hypothesis is used (seven times in total), the mere hypothesis of abelian recurrence on length-$2$ subwords (which also implies the recurrence of letters) suffices.
\end{proof}

It follows from Proposition~\ref{prop:CR_improved} that $w$ is not abelian recurrent on its length-$2$ subwords. Therefore, its induced word $I_2(w)$, which is an infinite $d$-ary word, cannot contain each of the $d$-letter infinitely many times. Thus, the frequency of at least one letter in $I_2(w)$ is zero. It then follows from Lemma~\ref{lemma:induction_basics} Assertion $3$ that 
\[\dim \Span_{\Q} (f_w(a): a \in \{1,\ldots,d\}) \leq d-1.\]
The proof of Theorem~\ref{th:Rauzy_conj-dary} is complete.

\subsection{Proof of Corollary~\ref{cor:extraction_lineaire}} \label{ss:proof_extraction}

Let $d \geq 3$. Let $w \in \{1,\ldots,d\}^{\N}$ be an infinite $d$-ary word admitting letter frequencies, and for which there exists $\ell \in \N_{>0}$ such that $\ac_w(\ell n) \leq d$ for all $n \in \N_{>0}$. We proceed by contradiction, and assume  that the letter frequencies of $w$ are rationally independent. By Proposition~\ref{prop:belowbound_abcomp}, we already have $\ac_w(\ell n) = d$ for all $n \in \N_{>0}$.

By the equality $\ac_w(\ell)=d$,  the induced word $w' = I_\ell(w)$ is written over a $d$-letter alphabet $\Acal_\ell$, and  its induction matrix $\bm{M}_\ell(w)$ is a square matrix.

\begin{lemma}\label{lemma:extraction}In this context, we have the  following three properties: \begin{enumerate}
\item[\emph{1.}] $\bm{M}_\ell(w) \in\GL_d(\Q)$;
\item[\emph{2.}]  The letter frequencies of $w'$ exist and are rationally independent;
\item[\emph{3.}] $\ac_{w'}(n) = d$ for all $n \in \N_{>0}$.
\end{enumerate}
\end{lemma}

\begin{proof}
It suffices to verify that in the proofs of Lemmas~\ref{lemma:GL3}, \ref{lemma:induced_frequencies_exist}, and \ref{lemma:induced_abconstant},   the assumption  \Hcomp: ``$w$ has a constant abelian complexity'' can be weakened into ``there exists $\ell\in \N_{>0}$ such that $\ac_w(\ell n) = d$ for all $n \in \N_{>0}$''. 
\end{proof}

It follows that $w'$ is an infinite $d$-ary word with constant abelian complexity and rationally independent letter frequencies. This contradicts Theorem~\ref{th:Rauzy_conj} (or Theorem~\ref{th:Rauzy_conj-dary} in the case $d\geq 4$).
The proof of Corollary~\ref{cor:extraction_lineaire} is complete.

\subsection{Proof of Corollary~\ref{coro:not-eventually-constant}} \label{ss:not_eventually_constant}

Let $d\geq 3$, and $w\in\{1,\ldots,d\}^{\N}$ be an infinite $d$-ary word. Assume that there exists $\ell \in \N_{>0}$ such that $\ac_w(n)=d$ for all $n\geq \ell$. Then we have in particular $\ac_w(\ell n)=d$ for all $n \in \N_{>0}.$ Therefore, by Corollary~\ref{cor:extraction_lineaire}, the letter frequencies of $w$ (which exist since the abelian complexity of $w$ is bounded) are rationally dependent. The proof of Corollary~\ref{coro:not-eventually-constant} is complete.


\small
\bibliography{Rauzy_biblio}
\bibliographystyle{alpha}

\end{document}